\documentclass[a4paper,12pt,oneside]{article}
\usepackage{amsmath,amsthm,amsfonts,amssymb,ifpdf}
\usepackage{cite}
\usepackage{amssymb}
\usepackage{mathrsfs}
\usepackage{amsmath}
\usepackage{hyperref}
\usepackage{amsthm}
\usepackage{graphicx}
\usepackage{epstopdf}
\usepackage{url}
\usepackage[all]{xy}
\usepackage{tikz-cd}
\usepackage{subcaption}
\usepackage{authblk}
\usepackage{enumerate}
\usepackage{enumitem}
\usepackage{stackengine}
\usepackage{accents}
\usepackage{stackengine}

\usepackage{stackengine,scalerel}

\newtheorem*{claim*}{Claim}

\newtheorem{theorem}{Theorem}[section]
\newtheorem{lemma}{Lemma}[section]
\newtheorem{corollary}{Corollary}[section]
\newtheorem{proposition}{Proposition}[section]
\newtheorem{remark}{Remark}[section]
\newtheorem{definition}{Definition}[section]

\numberwithin{equation}{section}

\newtheorem{example}{Example}[section]
\providecommand{\keywords}[1]
{
	\small
	\textbf{\textit{Keywords---}} #1
}
\providecommand{\msc}[1]
{
	\small
	\textbf{\textit{Mathematics Subject Classification (2020) ---}} #1
}

\title{\sc{Classes of Subgroups of a finite group with isomorphic quotient groups}}

\author {\sc{Tushar Halder}\thanks{24mmpp03@uohyd.ac.in} \ and \sc{ Sachin Ballal} \thanks{sachinballal@uohyd.ac.in (Corresponding author)}}

\affil{School of Mathematics and Statistics, University of Hyderabad, Hyderabad-500046, Telangana, India}

\date{}
\begin{document} 
	\maketitle
	\begin{abstract}
		In this paper, we study the poset $\text{QIso}(G)$ of classes of subgroups of a finite group $G$ having isomorphic quotients. Particularly, we study it for finite abelian groups and for some classes of non-abelian metacyclic groups, namely dihedral groups $D_n$, the group $U_{2np}$ and affine groups Aff($F$) over finite field $F$. In the case of abelian groups $G$, we show that there exists a duality from $\text{QIso}(G)$ to $\text{Iso}(G)$. We show that cyclic groups are the only groups for which $\text{QIso}(G)$, up to duality is isomorphic to its subgroup lattice. We also show that cyclic groups are the only nilpotent groups for which $\text{QIso}(G)$, up to duality is isomorphic to its lattice of normal subgroups. We characterize finite $2$-groups all of whose subgroups are q-self dual with $\text{QIso}(G)$ being a chain. We also characterize all $p$-groups $G$ of order at most $p^4$ whose $\text{QIso}(G)$ is a chain. \vspace{0.2cm}\\
		\\
		\msc{06A06, 06B99, 20D15, 20D30, 20D99, etc. }   
	\end{abstract}
	\section{Introduction}
The study of partially ordered set of certain classes of subgroups of a group has gained significant attention in the last decade. Mainardis in  \cite{Mainardis1998} characterized all finite groups $G$, for which there exists a poset isomorphism from the poset of conjugacy classes of subgroups of $G$ to that of some abelian group $A$. In \cite{tarnauceanu2015posetsclassesisomorphicsubgroups}, T\u{a}rn\u{a}uceanu studied the poset $\text{Iso}(G)$ of classes of isomorphic subgroups for several classes of groups, namely abelian groups, dihedral groups and ZM groups. It is evident that these posets encode informations about the group itself. So studying such posets can throw light on the structure of the group. Let $\mathcal{L}(G)$ denote the subgroup lattice of a finite group $G$. It is clear that $\mathcal{L}(G)$ is a complete bounded lattice with respect to inclusion, with trivial subgroup $\bigl<1\bigr>$ as the smallest element and $G$ as the largest element. The binary operations $\vee, \wedge$ on $\mathcal{L}(G)$ are defined as:
\begin{equation*}
	H \vee K = \bigl<H,K\bigr>, \hspace{0.2cm} H\wedge K = H \cap K, \hspace{0.2cm}\text{for all} \hspace{0.2cm} H,K \in \mathcal{L}(G).
\end{equation*}  
One of the important sublattice of $\mathcal{L}(G)$ is the lattice of all normal subgroups of $G$, denoted by $\mathcal{N}(G)$. Note that it forms a modular sublattice of $\mathcal{L}(G)$. For more information on subgroup lattices of group we refer \cite{Schmidt+1994}. \par
 Most of our notations are standard and is taken from \cite{isaacs2008finite} and \cite{Schmidt+1994}. A function $f$ between posets $L$ and $\bar{L}$ is called a duality if $f$ is an order reversing bijection. In this case, we say $L$ up to duality, is isomorphic to $\bar{L}$ and denote it by $L\cong \bar{L}^{\text{op}}$. We denote $C_n$ for chain of length $n-1$. If $x,y$ be two elements of a group $G$ then $x^y=y^{-1}xy$ and the commutator $[x,y]=x^{-1}y^{-1}xy$. Moreover, if $x_1, \dots , x_n \in G$, then $[x_1, x_2, \dots, x_n]$ defined inductively is $[[x_1, x_2, \dots, x_{n-1}],x_n]$. \par
 In \cite{tarnauceanusolitary}, T\u{a}rn\u{a}uceanu introduced the concept of solitary quotient as the dual concept of solitary subgroup. Following the same
technique, he construct a ``dual” for the set Iso($G$), that is \begin{equation}\label{exp1}
	\text{QIso}(G) = \{[H] \hspace{0.2cm}| \hspace{0.2cm}H \in \mathcal{N}(G)\}, \hspace{0.2cm}\text{where}\hspace{0.2cm} [H] = \{K \in \mathcal{N}(G) \hspace{0.2cm}|\hspace{0.2cm} G/K \cong G/H\}.
\end{equation}
and raised the question to endow the set given in (\ref{exp1}) with a suitable partial order relation and study the same. In an attempt to study $\text{QIso}(G)$, we have organized the paper as follows: In Section 2, we present some fundamental properties of the poset QIso($G$) associated to a finite group $G$. We show that for finite abelian groups $G$, the poset QIso($G$), up to duality is isomorphic to the poset Iso($G$). We note that the problem of $\text{QIso}(G)$ to be a lattice essentially reduces to the study of q-self dual groups $G$. We also study this poset for some metacyclic groups, namely dihedral groups, the group $U_{2np}$ and Affine groups $\text{Aff}(F)$. In Section 3, we characterize finite groups $G$ for which $\text{QIso}(G)$ up to duality, is isomorphic to $ \mathcal{L}(G)$ and finite nilpotent groups for which $\text{QIso}(G)$ up to duality, is isomorphic to $ \mathcal{N}(G)$. We characterize finite $2$-groups all of whose subgroups are q-self dual with $\text{QIso}(G)$ a chain and also characterize $p$-groups $G$ of order at most $p^4$ whose $\text{QIso}(G)$ is a chain.
	\section{The Poset $\text{QIso}(G)$}
	Let $G$ be a finite group and $\text{QIso}(G)$ be as in (\ref{exp1}). For $[H],[K]\in \text{QIso}(G)$, define $[H]\leq [K]$ if and only if there is a monomorphism from $G/K$ to $G/H$. Note that this relation is indeed a partial order relation on $\text{QIso}(G)$ and we call this poset $(\text{QIso}(G),\leq)$ as \emph{the poset of classes of subgroups with isomorphic quotients}. The question of studying the poset $\text{AutCl}(G)$ (see equation (\ref{autcl})) for a finite group $G$ was posed by T\u{a}rn\u{a}uceanu in \cite{tarnauceanu2015posetsclassesisomorphicsubgroups} and was studied by us in \cite{ballal2025posetsclassesautomorphicsubgroups}. It is defined as:
	\begin{equation*}
		\text{AutCl}(G)=\{[H]\hspace{0.1cm} |\hspace{0.1cm}H\in \mathcal{L}(G)\},\hspace{0.2cm}where \hspace{0.2cm} [H]=\{f(H)\hspace{0.1cm}| \hspace{0.1cm} f\in \text{Aut}(G)\} \hspace{0.2cm} \text{with}
	\end{equation*}
	\begin{equation}\label{autcl}
			[H_1]\lesssim [H_2] \hspace{0.1cm}\text{if and only if }\hspace{0.1cm} f(H_1)\subseteq H_2 \hspace{0.1cm}\text{for some}\hspace{0.1cm} f\in \text{Aut}(G).
	\end{equation}
	 \begin{remark}
	 Note that if $H, K\in \mathcal{N}(G)$, with $[H]=[K]$ in AutCl($G$), then $[H]=[K]$ in $\text{QIso}(G)$ but not conversely. For instance, in $G=\mathbb{Z}_2 \times \mathbb{Z}_4$, the subgroups $H=\mathbb{Z}_2 \times \bigl<(0,2)\bigr>$ and $K=\{0\} \times \mathbb{Z}_4$ serves as counterexample. 
	 \end{remark}
\begin{example}
Following are few examples of the poset $\text{QIso}(G)$:
\begin{enumerate}
		\item For any prime $p$, $\text{QIso}(\mathbb{Z}_p)\cong C_2$.
	\item For any prime $p$, $\text{QIso}(\mathbb{Z}_p \times \mathbb{Z}_p)\cong \text{QIso}(\mathbb{Z}_{p^2})\cong \text{QIso}(A_4)\cong C_3$.
	\item $\text{QIso}(\mathbb{Z}_2^3)\cong \text{QIso}(\mathbb{Z}_8) \cong \text{QIso}(D_4)\cong \text{QIso}(S_4)\cong C_4 $. 
	\item For $n\ge 5$, $\text{QIso}(A_n)\cong C_2$ and $\text{QIso}(S_n)\cong C_3$.
\end{enumerate}
\end{example}
In the following proposition, we estimate the width of the poset $ \text{QIso}(G)$ for a finite nilpotent group $G$.
	\begin{proposition}
		Let $G$ be a finite nilpotent group with number of distinct prime divisors of $|G|$ is $n$, then the width of $\text{QIso}(G)$ is at least $n$.
	\end{proposition}
	\begin{proof}
		Let $p_1,\dots, p_n$ be the distinct prime divisors of $|G|$. Since $G$ is nilpotent, $G$ is the direct product of its non-trivial Sylow subgroups, that is $G\cong P_1 P_2 \dots 
		P_n$, where $P_i\in Syl_{p_i}(G)$. Since each Sylow subgroup of $G$ is normal, we have $G/P_i \cong P_1 P_2\dots  P_{i-1} P_{i+1}\dots P_{n-1} P_n$, for all $1\le i\le n$. Certainly, the set $\{G/P_i : 1\le i\le n\}$ contains $n$ incomparable groups with respect to monomorphisms and hence $[P_i]$ are all distinct and incomparable in $\text{QIso}(G)$.
	\end{proof}
	\begin{corollary}
		Let $G$ be a finite nilpotent group with number of prime divisors of $|G|$ is $n$ and $\Phi(G)$ be its Frattini subgroup, then the width of the poset $\text{QIso}(G/\Phi(G))$ is at least $n$.
	\end{corollary}
	\begin{proof}
	Since $G$ is nilpotent, it follows that $G/\Phi(G)$ is abelian. Moreover, all Sylow $p$-subgroups of $\Phi(G)$ are normal in $G$. By Frattini argument, $G=N_G(P)\Phi(G)$, for any $P\in Syl_p(\Phi(G))$, so $G=N_G(P)$. Now suppose that $p$ divides $|G|$ but not $|G/\Phi(G)|$, then $\Phi(G)$ has a Sylow $p$-subgroup $P$ with $P\in \mathcal{N}(G)$. By Schur-Zassenhaus theorem (see Theorem 3.8 of \cite{isaacs2008finite}), there exists a complement $H\le G$ for $P$ in $G$. Now for a maximal subgroup $M$ of $G$ containing $H$, we have $p$ divides $|G/M|$ and therefore $p$ divides $|G/\Phi(G)|$, a contradiction. The result now follows from Proposition 2.1.
	\end{proof}
In the following theorem,  we show that for a finite abelian group $G$, the poset QIso($G$) is a distributive lattice.
	\begin{theorem}
		Let $G$ be a finite abelian group. Then the poset QIso($G$) up to duality, is isomorphic to the poset $\text{Iso}(G)$ and in particular, QIso($G$) is a distributive lattice. 
	\end{theorem}
	\begin{proof}
		Let $H$ be a subgroup of $G$, then $H$ is isomorphic to a quotient of $G$ and conversely every quotient of $G$ is isomorphic to a subgroup of $G$. This is because the map $G\longmapsto \text{Hom}(G, S^1)$, where $S^1$ is the circle group, is a contravariant exact functor in the category of abelian groups. Therefore it maps the exact sequence $0 \to H \to G \to G/H \to 0 $ to the sequence
		\begin{equation}\label{exp3}
			0 \to \text{Hom}(G/H,S^1) \to \text{Hom}(G,S^1)  \to \text{Hom}(H,S^1) \to 0.
		\end{equation}
		Moreover, as any finite abelian group $G$ is isomorphic to $\text{Hom}(G,S^1)$, so from (\ref{exp3}), it follows that $G/H$ is isomorphic to a subgroup of $G$.
		 Thus by realizing a class $[H]$ of QIso($G$) with the class $[H^*]$ in $\text{Iso}(G)$, where $H^* \le G$ with $H^*\cong G/H$, we have a one-to-one correspondance between elements of QIso($G$) and Iso($G$). It now remains to check that this one-to-one correspondance is a dual lattice isomorphism. So if $[H]\le [K]$ in QIso($G$), then there is a isomorphic copy of $G/K$ in $G/H$, therefore $[K^*]\le [H^*]$ in Iso$(G$). Conversely, if $[K^*]\le [H^*]$ in Iso($G$), then there is a monomorphism from $G/K$ to $G/H$, so $[H]\le [K]$ in QIso($G$). The distributivity of $\text{QIso}(G)$ follows from the fact that $\text{Iso}(G)$ is distributive (see Proposition 1 of \cite{tarnauceanu2015posetsclassesisomorphicsubgroups}).
	\end{proof}
	\begin{remark}
		Note that for a finite abelian group $G$, with $	G\cong  \prod_{i=1}^{k} G_i$, where $G_i$'s are prime power order. From \cite{tarnauceanu2015posetsclassesisomorphicsubgroups}, we have $	\text{QIso}(G)\cong \prod_{i=1}^{k} \text{QIso}(G_i).$
	\end{remark}
Note that in general $\text{Iso}(G)$ up to duality, need not be isomorphic to $\text{QIso}(G)$. In the following proposition we construct a family of non-abelian groups $G$ whose  $\text{Iso}(G)^{\text{op}}$ and $\text{QIso}(G)$ are isomorphic.
\begin{proposition}
	There exists an infinite family of non-abelian groups $\{G\}$ for which the poset $\text{Iso}(G)^{\text{op}}$ and $\text{QIso}(G)$ are isomorphic. 
\end{proposition}
\begin{proof}
	For any fixed odd prime $p$, consider the following subgroups of $GL(3,p)$:
	\begin{align*}
		N_p=\biggl\{\begin{pmatrix}
			1 & a & b\\ 0 & 1 & 0 \\ 0 & 0 & 1
		\end{pmatrix} : a,b \in \mathbb{Z}_p
		\biggr\}\hspace{0.2cm} \text{and} \hspace{0.2cm}
		H_p=\biggl\{\begin{pmatrix}
			1 & 0 & 0\\ 0 & 1 & c \\ 0 & 0 & 1
		\end{pmatrix} : c \in \mathbb{Z}_p
		\biggr\}.
	\end{align*}
	Let $G_p:= N_p \rtimes H_p$, where the action of $H_p$ on $N_p$ is via conjugation. Certainly,$\bigl<1\bigr>$, $N_p$ and $G_p$ are normal in $G_p$. Furthermore, note that the centre $Z(G_p)\cong \mathbb{Z}_p$ and $G_p/ Z(G_p)\cong \mathbb{Z}_p \times \mathbb{Z}_p$. It follows that $\text{QIso}(G_p)\cong C_4$. Since $G_p$ is a non-abelian group of order $p^3$ and exponent $p$, by Theorem 5 of \cite{tarnauceanu2015posetsclassesisomorphicsubgroups}, it follows that $\text{Iso}(G_p)$ is a chain. Since $G_p$ has subgroup of order $1,p,p^2$ and $p^3$, we have $\text{Iso}(G_p)\cong C_4$.
\end{proof}
\begin{definition} (see \cite{AN201135}) Let $G$ be a group (finite or infinite).
	\begin{enumerate}
		\item A group $G$ is s-self dual if every subgroup of $G$ is isomorphic to a quotient of $G$.
		\item A group $G$ is q-self dual if every quotient of $G$ is isomorphic to a subgroup of $G$.
		\item A group is self dual if it is both s-self dual and q-self dual.
	\end{enumerate}
\end{definition}
 The structure of s-self dual groups is studied in \cite{AN201135}. The study of $\text{QIso}(G)$ is inherently related with the study of q-self dual groups as the following proposition shows that $G$ being q-self dual is a necessary condition for $\text{QIso}(G)$ to be a lattice.
\begin{proposition}
	For a finite group $G$, if QIso($G$) is a lattice, then $G$ is a q-self dual group.  
\end{proposition}
\begin{proof}
	Let QIso($G$) be a lattice, then QIso($G$) has a minimum element $[M]$ as $G$ is finite. Certainly, $[M]\le [\bigl<1\bigr>]$, which implies there is a subgroup isomorphic to $G$ in $G/M$. As $|G/M|\le |G|$, this implies $G \cong G/M$ and therefore, $M=\bigl<1\bigr>$. So, $[\bigl<1\bigr>]\le [H]$, for all $[H]\in $ QIso($G$), which means $G/H$ is isomorphic to a subgroup of $G$, for any normal subgroup $H$ of $G$. So, $G$ is a q-self dual group. 
\end{proof}
 \begin{remark}
 	Note that q-self duality is not the sufficient condition for QIso($G$) to be a lattice. For instance, consider the q-self dual group $G=D_3 \times \mathbb{Z}_6$. Certainly, QIso($G$) is not a lattice as the class determined by $D_3 \times 0$ and $1 \times \mathbb{Z}_6$ does not possess a supremum. In contrast to s-self duality, q-self duality does not imply nilpotency.
 \end{remark}
 The following proposition asserts that the study of nilpotent q-self dual groups reduces to the study of q-self dual $p$-groups.
 \begin{proposition}
 	Let $G$ be a finite nilpotent group, then $G$ is a q-self dual group if and only if all its Sylow subgroups are q-self dual.
 \end{proposition}
 \begin{proof}
 	Let $|G|=p_1^{\alpha_1}\dots p_n^{\alpha_n}$ be the prime factorization of $|G|$, then $G=P_1P_2\dots P_n$, where $P_i\in Syl_{p_i}(G)$. Suppose $N$ is any normal subgroup of $G$, then $N$ itself is nilpotent. So $N=N_1N_2\dots N_n$, where $N_i \le P_i$. Now let $g_i\in P_i$, then for $g=g_1g_2\dots g_n\in G$, we have $N^g=N_1^{g_1}N_2^{g_2}\dots N_n^{g_n}=N$, so $N_i^{g_i}=N_i$, that is $N_i$ is normal in $P_i$. Therefore the quotient $G/N\cong (P_1/N_1)(P_2/N_2)\dots (P_n/N_n)\cong H_1H_2\dots H_n$, where $H_i$ is a subgroup of $P_i$ with $P_i/N_i\cong H_i$. So $G$ is a q-self dual group. \par Conversely, let $N_i$ be a normal subgroup of $P_i$, then since $G$ is a q-self dual group, so $P_i/N_i\cong G/(P_1\dots P_{i-1}N_iP_{i+1}\dots P_n)\cong H_i'$, for some $H_i'\le G$. Since $H_i'$ is a $p_i$-subgroup of $G$, so $H_i'\le P_i$ and consequently $P_i$ is q-self dual.
 \end{proof}
 		
\begin{definition}	Let $H,K$ be subgroups of $G$.
	\begin{enumerate}
		\item By maximal subgroup up to monomorphism in $H,K$, we mean a subgroup $S$ of $G$ with the property that there are monomorphisms from $S$ to $H,K$ and if $S'$ is any other subgroup with this property then there is a monomorphism from $S'$ to $S$.
		\item  By minimal subgroup up to monomorphism from $H,K$, we mean a subgroup $\tilde S$ of $G$ with the property that there are monomorphisms from $H,K$ to $\tilde S$ and if $S''$ is any other subgroup with this property then there is a monomorphism from $S$ to $S''$. 
	\end{enumerate}
\end{definition}
	\begin{theorem}
		Let $G$ be a finite self dual group with the property that any two subgroup $H,K$ of $G$ satisfies Definition 2.2, then the poset $\text{QIso}(G)$ is a lattice.
	\end{theorem}
	\begin{proof}
			Let $G$ be a self dual group and $[H],[K]$ be any two arbitrary elements of $\text{QIso}(G)$. Since $G$ is q-self dual, denote the quotients $G/H$ and $G/K$ by the subgroups of $G$, namely $S_H$ and $S_K$, respectively. Now consider maximal subgroup $\widetilde{S}$ up to monomorphism in $S_H$ and $S_K$ and minimal subgroup $\widehat S$ up to monomorphism from $S_H$ and $S_K$. Since $G$ is s-self dual, the subgroup of $G$ that corresponds to $\widetilde{S}$ and $\widehat S$ occurs as a quotient $G/\widetilde T$ and $G/ \widehat T$, respectively, for some normal subgroup $\widetilde{T}$ and $\widehat{T}$ of $G$. It now follows that $[H]\vee [K]=[\widetilde{T}]$ and $[H]\wedge [K]=[\widehat{T}]$, so $\text{QIso}(G)$ is a lattice.  
	\end{proof}
	We now study the poset $\text{QIso}(G)$ for some classes of finite metacyclic groups. One of the important non-abelian metacyclic group is the dihedral group $D_n$. The presentation of $D_n$ is as follows:
	\begin{equation*}
		D_n=\bigl<r,s\hspace{0.2cm}|\hspace{0.2cm} r^n=1=s^2, \hspace{0.2cm} s^{-1}rs=r^{-1}\bigr>.
	\end{equation*}
	\begin{theorem}
		For any positive integer $n$, the poset QIso($D_n$) is a distributive lattice containing $\tau(n)+1$ many elements.
	\end{theorem}
	\begin{proof}
			Note that when $n$ is odd, the only normal subgroups $D_n$ are $H_d:= \bigl<r^d\bigr>$, where $d$ is a divisor of $n$. For even $n$, in addition to $H_d$, there are two more normal subgroups of index $2$, namely $\bigl<r^2,s\bigr>$ and $\bigl<r^2,rs\bigr>$. From which we can infer that $|\text{QIso}(D_n)|= \tau(n) + 1$, where $\tau(n)$ is the number of all divisors of $n$. 
		By the above description of normal subgroups of $D_n$, if $n$ is even then it is clear that  $\bigl<r^2,s\bigr>$ and $\bigl<r^2,rs\bigr>$ are in the class of $H_1$. So all elements of QIso($D_n$) (except $[D_n]$), can be represented by $[H_d]$, for $d$ dividing $n$. Let $d_1$ and $d_2$ be divisors of $n$, then the subgroups $H_{d_1}$ and $H_{d_2}$ have index $2d_1$ and $2d_2$, respectively. If $d_1=1$, then $D_n/H_{d_1}\cong \mathbb{Z}_2$, else $D_n/H_{d_1}\cong D_{d_1}$. Certainly, $D_{d_1}$ is contained in $D_{d_2}$ if and only if $d_1$ divides $d_2$. Thus the join and meet in QIso($D_n$) is as follows
		\begin{equation}\label{Dih}
			[H_{d_1}]\vee [H_{d_2}]=[H_{\text{gcd}(d_1,d_2)}] \hspace{0.2cm} \text{and} \hspace{0.2cm}	[H_{d_1}]\wedge [H_{d_2}]=[H_{\text{lcm}(d_1,d_2)}].
		\end{equation}
		The distributivity of the lattice $\text{QIso}(D_n)$ follows from equation (\ref{Dih}). 
	\end{proof}
	\begin{remark}
		Theorem 2.3 shows that unlike $\mathcal{N}(D_n)$, QIso($D_n$) is always distributive. Also it is clear that QIso($D_p$) for odd prime $p$ is a chain of length 2.  
	\end{remark}
	\begin{proposition}
		Let $G$ be a split extension of $\mathbb{Z}_2$ by the normal subgroup $\mathbb{Z}_n$, where $\mathbb{Z}_n^{\times}$ is cyclic, then $\text{QIso}(G)$ is a distributive lattice.
		\begin{proof}
			Since $G$ is a split extension, we have $G\cong \mathbb{Z}_n \rtimes_{\phi} \mathbb{Z}_2$, for some homomorphism $\phi:\mathbb{Z}_2 \to \text{Aut}(\mathbb{Z}_n)$. Since $\text{Aut}(\mathbb{Z}_n) \cong \mathbb{Z}_n^{\times}$ is cyclic, so the only choices of $n$ are $1,2,4,p^k$ or $2p^k$, where $p$ is an odd prime. Now consider the following cases:
			\begin{enumerate}[label=(\roman*)]
				\item If $n=1$, then $\phi$ is trivial and hence $G\cong \mathbb{Z}_2$.
				\item If $n=2$, since $\text{Aut}(\mathbb{Z}_2)$ is trivial group, thus we have $\phi$ is trivial and hence $G \cong \mathbb{Z}_2 \times \mathbb{Z}_2$.
				\item If $n=4$, the only solution of $m^2\equiv 1 \hspace{0.1cm}\text{mod}\hspace{0.1cm}4$ are 1 and $-1$, therefore $G\cong \mathbb{Z}_4 \times \mathbb{Z}_2$ or $G\cong D_4$.
				\item If $n=p^k$, then $m^2\equiv 1 \hspace{0.1cm}\text{mod}\hspace{0.1cm}p^k$ reduces to $p^k|(m-1)(m+1)$. If $p$ divides both $m+1$ and $m-1$, then $p$ divides 2, which is not possible. So $m\equiv \pm 1 \hspace{0.1cm}\text{mod}\hspace{0.1cm}p^k $ and consequently $G\cong \mathbb{Z}_{p^n}$ or $G\cong D_{p^k}$.
				\item If $n=2p^k$, then the expression $m^2\equiv 1 \hspace{0.1cm}\text{mod}\hspace{0.1cm}n$ is equivalent to the system of equations $m^2\equiv 1 \hspace{0.1cm}\text{mod}\hspace{0.1cm}2$ and $m^2\equiv 1 \hspace{0.1cm}\text{mod}\hspace{0.1cm}p^k$. Since the former expression has a unique solution and the latter has two solutions, thus $m^2\equiv 1\hspace{0.1cm}\text{mod}\hspace{0.1cm}n$ has two solutions. Therefore $G \cong \mathbb{Z}_{2p^k}\times \mathbb{Z}_2$ or $D_{2p^k}$. 
			\end{enumerate} 
			The result now follows from Theorem 2.1 and 2.3.
		\end{proof}
	\end{proposition}
\begin{remark}
	We now provide two families of groups whose poset of classes of isomorphic quotients does not form a lattice. 
	\begin{enumerate}
 \item 	Consider the dicyclic group of order $4m$. For $m=1$, $Q_{4m}$ is isomorphic to $\mathbb{Z}_4$ and for $m\ge 2$, 
$Q_{4m}\cong \bigl< x,y \hspace{0.2cm}| \hspace{0.2cm} x^{2m}=1=y^4, \hspace{0.1cm} x^m=y^2, \hspace{0.1cm} \text{for} \hspace{0.1cm} g\notin \bigl<x \bigr>, g^{-1}xg=x^{-1} \bigr>.$ For $m\ge 2$, the center $Z(Q_{4m})$ of $Q_{4m}$ is $\{1,y^2\}=\{1,x^m\}$ and thus the quotient group $Q_{4m}/Z(Q_{4m})\cong D_m$. Certainly, $D_m$ can not be embedded as a subgroup in $Q_{4m}$, as $Q_{4m}$ has a unique element $x^m$ of order 2 whereas $D_m$ has at least $m$ many elements of order 2. So $Q_{4m}$ is not a q-self dual group for $m\ge 2$ and hence by Proposition 2.3, $\text{QIso}(Q_{4m})$ is not a lattice.
	\item Consider the minimal non-abelian $p$-group $M_p(n,n)= \bigl<a,b \mid a^{p^n}=b^{p^n}=1,\ [a,b]=a^{p^{n-1}}\bigr>$, where $n\ge 2$. Certainly $Z(M_p(n,n))=\bigl<a^p,b^p\bigr>$, so the quotient $M_p(n,n)/Z(M_p(n,n))$ $\cong M_p(n,1)$. However, no subgroup of $M_p(n,n)$ is isomorphic to $M_p(n,1)$ (see Lemma 3.7 of \cite{AN201135}).
	\end{enumerate}
\end{remark}
In \cite{shelash}, Shelash studied the metacyclic group $U_{2np}$, where $p$ is an odd prime. The presentation of $U_{2np}$ is $\bigl<a,b \hspace{0.2cm}|\hspace{0.2cm}a^{2n}=b^p=1, \hspace{0.1cm}aba^{-1}=b^{-1} \bigr>$.
\begin{lemma}(see \cite{shelash})
	The normal subgroup of the group $U_{2np}$ has one of the following forms:
	\begin{enumerate}
		\item All cyclic subgroups $\bigl<a^i\bigr>$ such that $ 2|i|2n$.
		\item All subgroups $\bigl<a^i,b\bigr>$ with $i|2n$.
	\end{enumerate}
\end{lemma}
\begin{theorem}
	Let $n\in \mathbb{N}$ and $p$ be an odd prime. The poset $\text{QIso}(U_{2np})$ is a distributive lattice.
\end{theorem}
\begin{proof}
In view of Lemma 2.1, we have the quotient $U_{2np}/\bigl<a^i\bigr>\cong \bigl<a^{\frac{2n}{i}},b\bigr> $ or $U_{2np}/\bigl<a^i,b\bigr>\cong \bigl<a^{\frac{2n}{i}}\bigr> $. Note that $U_{2np}$ is a q-self dual group. For even $i$, it is clear that there is a monomorphism from $U_{2np}/\bigl<a^i,b\bigr>$ to $U_{2np}/\bigl<a^i\bigr>$, so $[\bigl<a^i\bigr>]\le [\bigl<a^i,b\bigr>] $. It now follows that:
	\begin{enumerate}[label=(\roman*)]
	\item 	$[\bigl<a^i\bigr>]\vee[\bigl<a^j\bigr>]=[\bigl<a^{\gcd(i,j)}\bigr>],$
	\item $	[\bigl<a^i\bigr>] \wedge [\bigl<a^j\bigr>]=[\bigl<a^{\text{lcm}(i,j)}\bigr>],$
	\item $	[\bigl<a^i,b\bigr>] \vee [\bigl<a^j,b\bigr>]=[\bigl<a^{\gcd(i,j)},b\bigr>],$
	\item $	[\bigl<a^i,b\bigr>]\wedge [\bigl<a^j,b\bigr>]=[\bigl<a^{\text{lcm}(i,j)},b\bigr>],$
	\item $	[\bigl<a^i\bigr>]\vee [\bigl<a^j,b\bigr>]=[\bigl<a^{\gcd(i,j)},b\bigr>],$
	\item $	[\bigl<a^i\bigr>]\wedge [\bigl<a^j,b\bigr>]=[\bigl<a^{\text{lcm}(i,j)}\bigr>]$.
	\end{enumerate} 
	For the sake of completion, we show that $	[\bigl<a^i\bigr>]\vee [\bigl<a^j,b\bigr>]=[\bigl<a^{\gcd(i,j)},b\bigr>]$. Certainly $[\bigl<a^{\gcd(i,j)},b\bigr>]$ is an upper bound for $[\bigl<a^i\bigr>]$ and $[\bigl<a^j,b\bigr>]$. Now if $[B]$ be any other upper bound, then there are monomorphisms from $U_{2np}/B$ to $\bigl<a^{\frac{2n}{i}},b\bigr>$ and $\bigl<a^{\frac{2n}{j}}\bigr>$. It follows that $b\in B$, so $B=\bigl<a^m,b\bigr>$, for some $m$. Since $\bigl<b\bigr>$ is normal in $U_{2np}$ and $\bigl<a\bigr>\cap \bigl<b\bigr>=1$, so there is a monomorphism from $\bigl<a^{\frac{2n}{m}}\bigr>$ to $\bigl<a^{\frac{2n}{i}}\bigr>$. It now follows that $[\bigl<a^{\gcd(i,j)},b\bigr>]\le [B]$. The other cases follows similarly. The distributivity follows from the fact that $\text{QIso}(U_{2np})$ can be embedded in the divisor lattice of $4n$.
\end{proof}
Another class of metacyclic group is the affine groups over fields denoted by Aff($F$). A typical element of Aff($F$) is 
\begin{equation*}
\begin{pmatrix}
	a & b \\ 0 & 1
\end{pmatrix}, \text{where} \hspace{0.1cm}a \in F^{\times} \hspace{0.1cm} \text{and}\hspace{0.1cm} b \in F.
\end{equation*}
 Note that every non-trivial normal subgroup $N_H$ of Aff($F$) occurs as an inverse image of a normal subgroup $H$ of $F^{\times}$ under the natural projection map $\pi $ from $\text{Aff}(G)$ onto $F^{\times}$ given by 
\begin{equation*}
\begin{pmatrix}
a & b \\ 0 & 1 
\end{pmatrix} \longmapsto a. 
\end{equation*}
In the following theorem, we show that the poset QIso(Aff($F$)) is a distributive lattice.
\begin{theorem}
	Let $F$ be a finite field. Then the poset $\mathscr{A}$ obtained by removing the smallest element from QIso(Aff($F$)), is isomorphic to QIso($F^{\times}$) and in particular, QIso(Aff($F$)) is a distributive lattice.
\end{theorem}
\begin{proof}
	Observe that the map $N_H \longmapsto H$ is a one-to-one order preserving correspondance between the poset of non-trivial normal subgroups of Aff($F$) and the lattice of all normal subgroups of $F^{\times}$. Moreover, an element  \begin{equation}\label{normal}
	\begin{pmatrix}
		a & b \\ 0 & 1 
	\end{pmatrix}
	N_H \in \text{Aff}(F)/ N_H  \hspace{0.1cm} \text{can be written as}\hspace{0.1cm} \begin{pmatrix}
		a & 0 \\ 0 & 1 
	\end{pmatrix} .
	\begin{pmatrix}
		1 & a^{-1}b \\ 0 & 1
	\end{pmatrix} N_H= 	\begin{pmatrix}
	a & 0 \\ 0 & 1 
	\end{pmatrix} N_H.
	\end{equation} 
	Thus for a non-trivial normal subgroup $N_H$, in view of equation (\ref{normal}), we have the quotient
	\begin{equation*}
		\text{Aff}(F)/ N_H  \cong \biggl< \begin{pmatrix}
			x & 0\\ 0 & 1 
		\end{pmatrix} N_H \biggr>,
	\end{equation*} 
	where $F^{\times}=\bigl<x\bigr>$. It is clear that $|\text{Aff}(F)/ N_H|= \text{exp}_H(x)$, where $\text{exp}_H(x)$ is the exponent of $x$ in $H$. So if $N_K\in [N_H]$, then $\text{Aff}(F)/ N_K \cong \text{Aff}(F)/ N_H$ and as a result $ \text{exp}_K(x)=\text{exp}_H(x)$. Since, $H$ and $K$ are cyclic, we have $H=K$ and therefore $[N_H]$ is a singleton class. So, $\mathscr{A}\cong (\mathcal{N}(\text{Aff}(F))-\{1\})^{\text{op}}\cong \mathcal{N}(F^{\times})^{\text{op}}\cong \mathcal{L}(F^{\times})^{\text{op}}\cong \text{QIso}(F^{\times})$. This implies that $\mathscr{A}$ is distributive and hence QIso(Aff($F$)) is distributive.
\end{proof}
We can also study the poset $\text{QIso}(G)$ for certain infinite groups $G$. To do that, we restrict to those groups $G$ for which all of its non-trivial normal subgroups has finite index. Such classes of groups are called $\mathscr{D}_1$-groups and was studied by Wilson in \cite{wilson_1971}. We denote $\text{QIso}^{\times}(G)$ for the poset obtained from $\text{QIso}(G)$ by removing the class containing trivial subgroup.
\begin{theorem}
		Let $\text{Dih}_\infty$ denote the infinite dihedral group, then the poset $\text{QIso}^{\times}(\text{Dih}_\infty)$ is isomorphic to $\text{QIso}^{\times}(G)$, where $G$ is infinite cyclic group. And in particular, $\text{QIso}^{\times}(\text{Dih}_\infty)$ is an unbounded distributive lattice. 
\end{theorem}
\begin{proof}
	We first show that if $G$ is infinite cyclic, then the poset $\text{QIso}^{\times}(G)$ is isomorphic to a sublattice of the division lattice. Since $G$ is in the class $\mathscr{D}_1$, we have $\text{QIso}(G)$ is a poset. Let $T_{\infty}$ denote the divisor poset without the least element 0. Then map $\phi:\text{QIso}^{\times}(G)\to T_{\infty}$ given by $\phi([H])=|G:H|$ is an order preserving injection. Moreover in $\phi(\text{QIso}^{\times}(G))$, we have for $[H],[K]\in \text{QIso}^{\times}(G)$ that $\phi([H])\vee \phi([K])=\gcd(|G:H|,|G:K|)=\phi([H]\vee[K])$ and $\phi([H])\wedge \phi([K])=\text{lcm}(|G:H|,|G:K|)=\phi([H]\wedge [K])$, hence $\phi(\text{QIso}^{\times}(G))$ is a sublattice of the division lattice. \par 
	We recall that $\text{Dih}_\infty=\bigl<r,s\hspace{0.1cm}|\hspace{0.1cm} s^2=1, \hspace{0.1cm} s^{-1}rs=r^{-1} \bigr>\cong  \mathbb{Z}\rtimes \mathbb{Z}_2$, where $\mathbb{Z}_2$ inverts $\mathbb{Z}$. Note that all quotients of $\text{Dih}_\infty$ is either $D_n$, for some $n$ or $\text{Dih}_\infty$ itself. Moreover note that all finite dihedral groups appear as quotients of $\text{Dih}_\infty$. The result follows from the above discussion and the fact that $D_n $ is contained in $ D_m$ if and only if $n$ divides $m$.
\end{proof}
\section{Groups $G$ for which $\text{QIso}(G)$ is isomorphic to some standard lattices}
In this section, we characterize finite groups $G$ with $\text{QIso}(G)\cong \mathcal{L}(G)^{\text{op}}$ and finite nilpotent groups $G$ with $\text{QIso}(G)\cong \mathcal{N}(G)^{\text{op}}$. We establish that the only finite $2$-groups all of whose subgroups are q-self dual with $\text{QIso}(G)$ being a chain are either cyclic 2-group, an elementary abelian 2-group or the dihedral group $D_{2^n}$, $n\ge 2$. Lastly, we characterized all $p$-groups $G$ of order at most $p^4$ whose $\text{QIso}(G)$ is a chain. \par 
We now discuss a result concerning groups $G$ with QIso($G$) being isomorphic to $C_3$ and all non-trivial quotients being $\mathbb{Z}_p$. 
\begin{proposition}
	Let $G$ be a finite group with $\text{QIso}(G)\cong C_3$ and if one of its normal subgroup is isomorphic to $\mathbb{Z}_p$ and all the nontrivial quotients of $G$ is isomorphic to $\mathbb{Z}_q$, then $G$ is either $\mathbb{Z}_{p^2}$, $\mathbb{Z}_p \times \mathbb{Z}_p$ or $\mathbb{Z}_p \rtimes \mathbb{Z}_q$, with $q$ dividing $p-1$.
\end{proposition}
\begin{proof}
	Firstly we need to find all groups $G$, which satisfies the short exact sequence
	\begin{equation}\label{exp2}
		1 \to \mathbb{Z}_p \to G \to \mathbb{Z}_q \to 1. 
	\end{equation}
	and then restrict to all $G$'s with QIso($G$) being isomorphic to $C_3$. If $p=q$, then it is clear that $G$ is isomorphic to $\mathbb{Z}_{p^2}$ or $\mathbb{Z}_p \times \mathbb{Z}_p$, as these are only groups of order $p^2$. If $p\neq q$, then the short exact sequence in (\ref{exp2}) splits and as a consequence $G\cong \mathbb{Z}_p \rtimes \mathbb{Z}_q$. If $q$ does not divide $p-1$, then $G \cong \mathbb{Z}_{pq}$. Clearly, QIso($\mathbb{Z}_{pq}$) is not $C_3$. If $q$ divides $p-1$, then any subgroup of order $q$ in $G$ is not normal. So, QIso($G$) is isomorphic to $C_3$.
\end{proof}
In the following theorem we show that cyclic groups are the only groups for which $\text{QIso}(G)$ is isomorphic to $\mathcal{L}(G)^{\text{op}}$.
\begin{theorem}
	The poset $\text{QIso}(G)$ is isomorphic to $\mathcal{L}(G)^{\text{op}}$ if and only if $G$ is cyclic.
\end{theorem}
\begin{proof}
	Suppose that $\text{QIso}(G)$ is isomorphic to $\mathcal{L}(G)^{\text{op}}$. Since  $|\text{QIso}(G)|\le |\mathcal{N}(G)|\le |\mathcal{L}(G)|$, we must have $|\mathcal{N}(G)|=|\mathcal{L}(G)|$ and all classes in $\text{QIso}(G)$ are singleton. So all subgroups of $G$ are normal in $G$ and so $G$ is a Dedekind group. If $G$ is abelian, then by Theorem 2.3, $\text{QIso}(G)\cong \text{Iso}(G)^{\text{op}}$ and therefore, $|\text{Iso}(G)|=|\mathcal{L}(G)|$, that is all classes in $\text{Iso}(G)$ are singleton. If $G$ is non-cyclic then there is a subgroup $H$ of $G$ isomorphic to $\mathbb{Z}_p \times \mathbb{Z}_p$, for some prime $p$, which is not possible, so $G$ must be cyclic. If $G$ is non-abelian, then $G=Q_8 \times A \times E$, where A is an abelian odd ordered group and $E$ is an elementary abelian 2-group. If $Q_8$ appears in the decomposition of $G$, then $G/({\bigl<i\bigr>\times A\times E})\cong \mathbb{Z}_2 \cong G/({\bigl<j\bigr>\times A\times E})$, which is not possible as all classes in $\text{QIso}(G)$ are singleton. This forces $G$ to be abelian, a contradiction.
\end{proof}
\begin{corollary}
	Let $G$ be a finite nilpotent group. The poset $\text{QIso}(G)$ up to duality, is isomorphic to the lattice of subnormal subgroups $\mathcal{R}(G)$ if and only if $G$ is cyclic.
\end{corollary}
\begin{proof}
	Since for a finite nilpotent group $G$ every subgroup is subnormal, we have $\mathcal{R}(G)=\mathcal{L}(G)$. The result now follows from Theorem 3.1.
\end{proof}
We now show that cyclic groups are the only finite nilpotent groups with $\text{QIso}(G)$ is isomorphic to $\mathcal{N}(G)^{\text{op}}$. To do that, we essentially need the following theorems.
\begin{theorem} (see \cite{kulakoff})
	Let $G$ be a non- cyclic group of order $p^n$, $p$ being an odd prime, the number of subgroup of $p^k$, for $0<k<n$, is congruent to $1+p$ (mod $p^2$).
\end{theorem}
\begin{definition} (see \cite{easterfield1938extension})
	A group is said to be of exceptional type if it is one of the following.
	\begin{enumerate}
		\item A cyclic group.
		\item The dihedral group of order $2^n$ ($n\ge 3$).
		\item The quaternion group of order $2^n$ ($n\ge 3$).
		\item The ``mixed" group of order $2^n$, generated by $a$ and $b$, where $a^{2^{n-1}}=1$, $b^2=1$, $b^{-1}ab=a^{-1+2^{k-2}}$ ($k\ge 4$).
	\end{enumerate}
\end{definition}
\begin{theorem} (see \cite{easterfield1938extension})
	If $G$ is of order $2^n$ and not exceptional, then the number of subgroups of $G$ of order $p^k$, where $0<k<n$, is congruent to 3 (mod 4).
\end{theorem}
\begin{theorem}
	Let $G$ be a finite nilpotent group. The poset $\text{QIso}(G)$ up to duality, is isomorphic to $\mathcal{N}(G)$ if and only if $G$ is cyclic.
\end{theorem}
\begin{proof}
	We first show that the only $p$-group $G$ for which the poset $\text{QIso}(G)$ is isomorphic to $\mathcal{N}(G)^{\text{op}}$ are cyclic. So suppose that $G$ is a non-cyclic $p$-group, ($p>2$), of order $p^n$, then by Theorem 3.2, the number of subgroups of order $p^{n-1}$ is at least 2 and all of them are normal in $G$. Therefore, $|\text{QIso}(G)|< |\mathcal{N}(G)|$. Now suppose that $p=2$, then by Theorem 3.3, the number of subgroups of order $p^{n-1}$ is at least 2 which again implies that $|\text{QIso}(G)|< |\mathcal{N}(G)|$, except for the exceptional groups $G$ in Definition 3.1. Thus it remains to check for the exceptional groups. For cyclic groups $G$, it is clear that $\text{QIso}(G)\cong \mathcal{N}(G)^{\text{op}}$. If $G$ is a dihedral groups of order $2^n$, generated by $r$ and $s$, with $r^{2^{n-1}}=1$, $s^2=1$, $s^{-1}rs=r^{-1}$ and $n\ge 3$, then we have normal subgroups $\bigl<r^2,s\bigr>$ and $\bigl<r^2,rs\bigr>$ of index 2, so $|\text{QIso}(G)|< |\mathcal{N}(G)|$. If $G$ is a quaternion group of order $2^n$ generated by $x$ and $y$ with $x^{2^{n-1}}=1$, $y^2=x^{2^{n-2}}$, $y^{-1}xy=x^{-1}$ and $n\ge 3$, then for $n\ge 4$, the normal subgroups $\bigl<x^2,y\bigr>$ and $\bigl<x^2,xy\bigr>$ are of index 2 and for $Q_8$, certainly $\bigl<i\bigr>$ and $\bigl<j\bigr>$ are index 2 subgroups. So, we have that $|\text{QIso}(G)|< |\mathcal{N}(G)|$. Now suppose that $G$ is the ``mixed" group of order $2^n$, generated by $a$ and $b$, where $a^{2^{n-1}}=1$, $b^2=1$, $b^{-1}ab=a^{-1+2^{k-2}}$ with ($k\ge 4$), then the subgroups $\bigl<a\bigr>$ and $\bigl<a^2,b\bigr>$ both are normal in $G$ and is of index 2, which implies $|\text{QIso}(G)|< |\mathcal{N}(G)|$. Thus if $G$ is a $p$-group with $\text{QIso}(G)\cong \mathcal{N}(G)^{\text{op}}$, then $G$ must be cyclic.\par 
	Suppose that $G$ is a nilpotent group and $\text{QIso}(G)\cong \mathcal{N}(G)^{\text{op}}$, then $G$ is a direct product of its non trivial Sylow subgroups, that is $G=P_1P_2\dots P_r$, where $P_i\in Syl_{p_i}(G)$. To show that $G$ is cyclic, it is sufficient to show that each of $P_i$'s are cyclic. Assume for the sake of contradiction that $P_1$ is not cyclic, then by the above paragraph, there are two distinct normal subgroups $N_1$ and $N_1'$ of $P_1$ of index $p_1$. So, $G/(N_1P_2P_3\dots P_r)\cong P_1/N_1\cong  \mathbb{Z}_{p_1}\cong P_1/N_1' \cong  G/(N_1'P_2P_3\dots P_r)$, so $[N_1P_2P_3\dots P_r]=[N_1'P_2P_3\dots P_r]$, which implies $|\text{QIso}(G)|< |\mathcal{N}(G)|$, a contradiction.
\end{proof}
\begin{lemma}
	Let $G=A \bigl<t\bigr>$ be the generalized dihedral group, where $A$ is an abelian group, $t^2=1$ and $t^{-1}at=a^{-1}$, for all $a\in A$, then 
	\begin{enumerate}
		\item A subgroup $N$ of $G$ with $N\subseteq A$ is normal in $G$ and quotient $G/N$ is isomorphic to $(A/N)\bigl<t\bigr>$.
		\item A subgroup $N \nsubseteq A$ is normal in $G$ if and only if $N=H\cup Ht$, where $H=N\cap A$ is a subgroup of $A$ containing $A^2$. Moreover the quotient $G/M$ is isomorphic to $A/H$, which is an elementary abelian 2-group.
	\end{enumerate}
\end{lemma}
\begin{proof}
	Let $N$ be a subgroup of $G$ that is contained in $A$, then $t^{-1}Nt=N^{-1}=N$, so $N$ is normal in $G$. Moreover the map $\pi: A \bigl<t\bigr>\to (A/N)\bigl<t\bigr>$ given by $\pi(a)=aN$ and $\pi(t)=t$, for all $a\in A$, is a surjective homomorphism with kernel $N$. \par
	Now let $N$ be a normal subgroup of $G$ with $N\nsubseteq A$ and set $H:= N\cap A$. Since $G=NA$, so we must have $2=|G:A|=|N:N\cap A|$. This implies $N=H\cup Hx_0$, for some $x_0\in N-A$. Write $x_0=a_0t$, for some $a_0\in A$. Let $h\in H$ and $a\in A$, then $a^{-1}ha=h\in H$ and $a^{-1}(hx_0)a=a^{-1}(ha_0t)a=ha_0a^{-1}ta=ha^{-2}x_0$. Since $N$ is normal in $G$, we must have $ha^{-2}x_0\in Hx_0$, that is $a^2\in H$, for all $a\in A$. So if $N=H\cup Ht$, where $A^2\subseteq H$, then the map $\phi: G\to A/H$ given by $\phi(a)=aH$ and $\phi(t)=H$ is a surjective homomorphism with kernel $N$. Since $A^2\subseteq H$, we have $A/H$ is elementary abelian 2-group.
\end{proof}
\begin{theorem}
	Let $G$ be a finite 2-group all of whose subgroups are q-self dual, then the poset $\text{QIso}(G)$ is a chain if and only if $G$ is a cyclic 2-group, or an elementary abelian 2-group or the dihedral group $D_{2^n}$ with $n\ge 2$.
\end{theorem}
\begin{proof}
	If $G$ is abelian 2-group, then certainly all subgroups of $G$ are q-self dual. From Theorem 2.1, we have $\text{QIso}(G)\cong \text{Iso}(G)^{\text{op}}$ and from Theorem 5 of \cite{tarnauceanu2015posetsclassesisomorphicsubgroups}, it follows that $G$ is either a cyclic $2$- group or an elementary abelian 2-group. Now if $G$ is a non-abelian 2-group all of whose subgroups are q-self dual, then by Theorem 3 of \cite{Janko} we have $G$ is either generalized dihedral group with maximal abelian subgroup $A$ is of exponent at least $3$ or $G=M_{2^{e+1}}\times V$, where $V$ is an abelian group with $\text{exp}(V)\le 2$. Here the group $M_{2^{e+1}}$ is given by the following presentation
	\begin{equation*}
		\bigl<a,b \mid a^{2^e}=1,\ b^2=1, \ [a,b]=a^{2^{e-1}}, \ e\ge 3\bigr>.
	\end{equation*}  
	So if $G=M_{2^{e+1}}\times V$, then the quotient of $G$ by normal subgroups $\bigl<a^2\bigr>\times V$ and $\bigl<a^4,b\bigr>\times V$ are isomorphic to $\mathbb{Z}_2\times \mathbb{Z}_2$ and $\mathbb{Z}_4$, respectively. Therefore $[\bigl<a^2\bigr>\times V]$ and $[\bigl<a^4,b\bigr>\times V]$ are incomparable in $\text{QIso}(G)$. Now if $G$ is a generalized dihedral group with $A$ being its maximal abelian subgroup with $\text{exp}(A)>2$, then $G\cong A\bigl<t\bigr>$, with $t^2=1$ and $t^{-1}at=a^{-1}$, for all $a\in A$. If $A$ is non-cyclic, then by Lemma 3.1, there exists normal subgroups $N_1$ and $N_2$ of $G$ with $N_1\le A$ such that $G/N_1\cong \mathbb{Z}_4 \rtimes \mathbb{Z}_2$ and $G/N_2 \cong \mathbb{Z}_2 \times \mathbb{Z}_2 \times \mathbb{Z}_2$. Thus $[N_1]$ and $[N_2]$ are incomparable in $\text{QIso}(G)$. If $A$ is cyclic 2-group, then $G\cong D_{2^n}$, since $G$ is non-abelian, we have $n\ge 2$, it follows from Theorem 2.3 that $\text{QIso}(D_{2^n})$ is a chain.
\end{proof}
In order to characterize all $p$-groups of order $\le p^4$ with $\text{QIso}(G)$ being chain, we essentially need the following results.
\begin{lemma}(see \cite{Burnside_2012})
	Up to isomorphism, there are two non-abelian groups of order $p^3$, namely 
	\begin{enumerate}
		\item  $G_1=\bigl<a,b \hspace{0.1cm}|\hspace{0.1cm} a^{p^2}=1,\hspace{0.1cm} b^p=1, \hspace{0.1cm} [a,b]=a^p\bigr>$ and 
		\item $G_2=\bigl<a,b,c \hspace{0.1cm}|\hspace{0.1cm} a^p=1,\hspace{0.1cm} b^p=1, \hspace{0.1cm} c^p=1,\hspace{0.1cm} c=[a,b], \hspace{0.1cm} ca=ac,\hspace{0.1cm} cb=bc\bigr>$.
	\end{enumerate} 
	Moreover, there are ten non-abelian groups of order $p^4$, namely
	 \begin{enumerate}
	 	\item  $P_1=\bigl< a,b \mid a^{p^3}=b^p=1,\ [a,b]=a^{p^2}\bigr>$,
	 	\item $P_2=\bigl< a,b,c \mid
	 	a^{p^2}=b^p=c^p=[a,b]=[a,c]=1,\ [b,c]=a^p
	 	\bigr>$,
	 	\item $P_3=\bigl< a,b \mid a^{p^2}=b^{p^2}=1,\hspace{0.1cm} [a,b]=a^p\bigr>$,
	 	\item$ P_4=\bigl< a,b,c \mid
	 	a^{p^2}=b^p=c^p=[a,b]=[b,c]=1,\ [a,c]=a^p
	 	\bigr>$,
	 	\item $P_5=\bigl< a,b,c,d \mid a^p=b^p=c^p=d^p=1, \ [c,d]=a, \ [b,d]=[a,d]=[b,c]=[a,c]=[a,b]=1\bigr>$,
	 	\item $P_6=\bigl< a,b,c \mid a^{p^2}=b^p=c^p=1, \ [a,c]=b, \ [a,b]=[b,c]=1\bigr>$,
	 	\item $P_7= \bigl< a,b,c \mid a^{p^2}=b^p=c^p=1, \ [a,b]=a^p, \ [a,c]=b, \ [b,c]=1\bigr>$,
	 	\item $P_8= \bigl< a,b,c \mid a^{p^2}=b^p=1, \ c^p=a^p, \ [a,b]=a^p, \ [a,c]=b, \ [b,c]=1\bigr>$,
	 	\item  $P_9= \bigl< a,b,c \mid a^{p^2}=b^p=1, \ c^p=a^{\alpha p}, \ [a,b]=a^p, \ [a,c]=b, \ [b,c]=1\bigr>$, where $\alpha$ is a non-residue modulo $p$, and
	 	\item $P_{10}= \bigl< a,b,c,d \mid a^p=b^p=c^p=d^p=1, \ [c,d]=b, \ [b,d]= a,\ [a,d]=[b,c]=[a,c]=[a,b]=1 \bigr>$, $p \ge 5$. 
	 \end{enumerate}
\end{lemma}
\begin{lemma}  (see \cite{ahmadi20132}) 
	Let $G$ be a finite nilpotent group of class 3, then for any $x,y\in G$ and $n\in \mathbb{N}$ we have
	\begin{equation*}
			(xy)^n= x^ny^n[y,x]^{\binom{n}{2}}[y,x,x]^{\binom{n}{3}}[y,x,y]^{\frac{n(n-1)(2n-1)}{6}}.
	\end{equation*}
\end{lemma}
\begin{lemma} (see \cite{Schmidt+1994})
	Let $G$ be a group with $x,y,z\in G$ and let $m\in \mathbb{Z}$, $n\in \mathbb{N}$, then
	\begin{enumerate}
		\item $[xy,z]=[x,z]^y[y,z]$ and $[x,yz]=[x,z][x,y]^z$.
		\item $[x^m,y]=[x,y]^m$ if $[x,y]$ and $x$ commutes.
		\item $[x,y^m]=[x,y]^m$ if $[x,y]$ and $y$ commutes.
		\item $(xy)^n=x^ny^n [y,x]^{\binom{n}{2}}$ if $[x,y]$ commutes with $x$ and $y$.
	\end{enumerate}
\end{lemma}
The following result can be deduced from Lemma 3.4 by applying induction on the power of $y$.
\begin{lemma}
	Let $G$ be any group with $x,y\in G$ and $j\in \mathbb{N}$, then 
	\begin{equation*}
		[x,y^j]=\prod_{i=0}^{j-1}[x,y]^{y^{i}}.
	\end{equation*}
\end{lemma}
\begin{theorem}
	Let $G$ be a non-abelian $p$-group with $|G|\le p^4$ ($p$ is odd). Then the poset $\text{QIso}(G)$ is a chain if and only if $G$ is isomorphic to one of the following groups:
	\begin{enumerate}
		\item $\bigl<a,b \hspace{0.1cm}|\hspace{0.1cm} a^{p^2}=1,\hspace{0.1cm} b^p=1, \hspace{0.1cm} [a,b]=a^p\bigr>$.
		\item $\bigl<a,b,c \hspace{0.1cm}|\hspace{0.1cm} a^p=1,\hspace{0.1cm} b^p=1, \hspace{0.1cm} c^p=1,\hspace{0.1cm} c=[a,b], \hspace{0.1cm} ca=ac,\hspace{0.1cm} cb=bc\bigr>$.
		\item $\bigl< a,b,c \mid a^{9}=b^3=c^3=1, \ [a,b]=a^3, \ [a,c]=b, \ [b,c]=1\bigr>$.
	\item $\bigl< a,b,c \mid a^{p^2}=b^p=1, \ c^p=a^p, \ [a,b]=a^p, \ [a,c]=b, \ [b,c]=1\bigr>$.
	\item  $\bigl< a,b,c \mid a^{p^2}=b^p=1, \ c^p=a^{\alpha p}, \ [a,b]=a^p, \ [a,c]=b, \ [b,c]=1\bigr>$, where $\alpha$ is a non-residue modulo $p$, $p\ge 5$.
	\item $\bigl< a,b,c,d \mid a^p=b^p=c^p=d^p=1, \ [c,d]=b, \ [b,d]= a,\ [a,d]=[b,c]=[a,c]=[a,b]=1 \bigr>$, $p \ge 5$. 
	\end{enumerate}
\end{theorem}
\begin{proof}
 Consider groups of order $p^3$ as in Lemma 3.2. We claim that $\text{QIso}(G_1)$ is a chain $C_4$, because there are no normal subgroups $N$ of $G_1$ with $G_1/N$ is a cyclic group of order $p^2$. On the contrary assume $N$ is such a subgroup, then $|N|=p$. As $G_1/N$ is cyclic, so we have that the derived subgroup $G_1'$ is contained in $N$. Since $Z(G_1)=\bigl<a^p\bigr>=\bigl<[a,b]\bigr>=G_1'$, we have $N=Z(G_1)$. Clearly, $p^2=\text{exp}(G_1/N)=\text{exp}(G/Z(G_1))=p$, a contradiction. Moreover, $G_1/Z(G_1)\cong \mathbb{Z}_p\times \mathbb{Z}_p \cong \bigl<a^p,b\bigr>\le G_1$, so $G_1$ is a q-self dual group. It follows that $\text{QIso}(G_1)$ is a chain. \par Now for the group $G_2$, all quotients of $G_2$ has exponents at most $p$ and by Corollary 7.2 of \cite{AN201135}, $G_2$ is a q-self dual group, so $\text{QIso}(G_2) $ is a chain $C_4$. \par
   Now consider groups of order $p^4$ as in Lemma 3.2: 
 \begin{enumerate}[label=(\roman*)]
 \item For the group $P_1$, we have $Z(P_1)=\bigl<a^p\bigr>$, so $P_1/Z(P_1)\cong \mathbb{Z}_p \times \mathbb{Z}_p$. Also since $\bigl<a^{p^2},b\bigr>$ is normal in $P_1$, we have $P_1/\bigl<a^{p^2},b\bigr>\cong \bigl<a\bigr>/\bigl<a^{p^2}\bigr>\cong \mathbb{Z}_{p^2}$, so $\text{QIso}(P_1)$ is not a chain.
 \item For $P_2$, we show that $\text{QIso}(P_2)$ is not a q-self dual. Note that $\bigl<a\bigr>$ is central in $P_2$ and the derived subgroup $P_2'=\bigl<[b,c]\bigr>=\bigl<a^p\bigr>$. So $P_2/P_2'\cong \mathbb{Z}_p\times \mathbb{Z}_p \times \mathbb{Z}_p$ but there are no subgroups of $P_2$ which is isomorphic to $\mathbb{Z}_p\times \mathbb{Z}_p \times \mathbb{Z}_p$. For if $H$ is such a subgroup of $P_2$, then $H/(H\cap \bigl<a\big>)\cong H\bigl<a\bigr>/\bigl<a\bigr>\le P_2/\bigl<a\bigr>$. So, $|H/(H\cap \bigl<a\bigr>)|=1,p$ or $p^2$. Certainly, if $|H/(H\cap \bigl<a\bigr>)|=1$, then $H=H\cap \bigl<a\bigr>$, which implies $H\subseteq\bigl<a\bigr>$, which is not possible. If $|H/(H\cap \bigl<a\bigr>)|=p$, then $|H\cap \bigl<a\bigr>|=p^2$. So $H\cap \bigl<a\bigr>=\bigl<a\bigr>$, and hence $\bigl<a\bigr>\subseteq H$, which is not possible as $p=\text{exp}(H)< \text{exp}(\bigl<a\bigr>)=p^2$. Lastly, if $|H/(H\cap \bigl<a\bigr>)|=p^2$, then $|H\cap \bigl<a\bigr>|=p$, which implies $H\cap \bigl<a\bigr>=\bigl<a^p\bigr>$ and therefore $P_2=H\bigl<a\bigr>$, which is not possible as $P_2$ is non-abelian.
 \item For the group $P_3$, note that $a^p\in P_3'$, so $P_3'$ is non-trivial and $[P_3',P_3]=\bigl<1\bigr>$. So, the nilpotency class of $P_3$ is 2 and therefore by Lemma 3.4, $[a,b^p]=[a,b]^p=1$, which implies $Z(P_3)=\bigl<a^p,b^p\bigr>$. We now have $P_3/Z(P_3)\cong \mathbb{Z}_p \times \mathbb{Z}_p$. Since $\bigl<a\bigr>$ is normal in $P_3$, we have $P_3/\bigl<a\bigl>\cong \mathbb{Z}_{p^2}$. Thus the poset $\text{QIso}(P_3)$ is not a chain. 
 \item In $P_4$, note that $P_4'=\bigl<[a,c]\bigr>=\bigl<a^p\bigr>$, so $P_4/P_4'\cong \mathbb{Z}_p\times \mathbb{Z}_p\times \mathbb{Z}_p$. Moreover, since $P_4=\bigl<b\bigr>\times \bigl<a,c\bigr>$ (see \cite{Burnside_2012}) and since $\bigl<b\bigr>$ is normal in $P_4$, we have $P_4/\bigl<b\bigr>\cong \bigl<a,c\bigr>\cong G_1$. So, $\text{QIso}(P_4)$ is not a chain.
 \item Note that $P_5=\bigl<b\bigr>\times \bigl<a,c,d\bigr>$ (see \cite{Burnside_2012}), so $\bigl<b\bigr>$ is normal in $P_5$ with $P_5/\bigl<b\bigr>\cong G_2$. Moreover $\bigl<a\bigr>$ is normal in $\bigl<a,c,d\bigr>$, so $\bigl<1\bigr>\times \bigl<a\bigr> $ is normal in $\bigl<b\bigr>\times \bigl<a,c,d\bigr>$, with quotient as $\mathbb{Z}_p\times \mathbb{Z}_p\times \mathbb{Z}_p$. So $\text{QIso}(P_5)$ is not a chain. 
 \item In $P_6$, clearly $\bigl<b\bigr>$ is normal in $P_6$ with $P_6/\bigl<b\bigr>\cong \mathbb{Z}_{p^2}\times \mathbb{Z}_p$. Also, $\bigl<a^p\bigr>$ is normal in $P_6$ with $P_6/\bigl<a^p\bigr>\cong G_2$, so $\text{QIso}(P_6)$ is not a chain.
 \end{enumerate}
 In $P_s$ ($s=7,8,9$), $[a,b]$ commutes with both $a$ and $b$, and therefore $(ab)^p=a^pb^p[b,a]^{\binom{p}{2}}=a^p$. Which implies that $[a^p,c]=a^{-p}(ab)^p=1$, and hence $\bigl<a^p\bigr>\le Z(P_s)$. For integers $i,j,k$ let $g=a^ib^jc^k\in Z(P_s)$, then $1=[g,b]=[a^ib^jc^k,b]=[a^i,b]^{b^jc^k}[b^jc^k,b]=([a,b]^i)^{b^jc^k}=(a^{pi})^{b^jc^k}$. Thus $p$ divides $i$, so write $i=pr$, for some $r$. Now $1=[g,a]=[a^{pr}b^jc^k, a]=[b^jc^k,a]=[b^j,a]^{c^k}[c^k,a]=(a^{-pj})^{c^k}b^{-k}$. Since $\bigl<a\bigr>\cap \bigl<b\bigr>=1$, which implies $p$ divides both $j$ and $k$, and therefore $Z(P_s)=\bigl<a^p\bigr>$. It now follows from Theorem 1.19 of \cite{isaacs2008finite} that $Z(P_s)$ is the unique normal subgroup of order $p$. Observe that $P_s/Z(P_s)\cong G_2$. Now the derived subgroup $P_s'=\bigl<a^p,b\bigr>$, which implies $P_s/P_s' \cong P_s'\le P_s$. Let $N$ be any normal subgroup of $P_s$ of order $p^2$, then by Theorem 1.19 of \cite{isaacs2008finite}, $\bigl<a^p\bigr>\le N$. Since $N/\bigl<a^p\bigr>$ is normal in $P_s/ \bigl<a^p\bigr>$ and $Z(P_s/ \bigl<a^p\bigr>)=\bigl<b\bigl<a^p\bigr>\bigr>$, we must have $N/\bigl<a^p\bigr>=\bigl<b\bigl<a^p\bigr>\bigr>$, so $N=\bigl<a^p,b\bigr>=P_s'$. This implies that $P_s'$ is the unique normal subgroup of $P_s$ of order $p^2$. Therefore the poset $\text{QIso}(P_s)$ is a chain if and only if there is a subgroup of $P_s$ which is isomorphic to $G_2$. 
 \begin{enumerate}[label=(\roman*), start=7]
 \item Note that the nilpotency class of $P_7$ is 3, so by Lemma 3.3, we have for $p\ge 5$ that 
 \begin{equation}\label{HP_1}
 	(a^ib^jc^k)^p=a^{ip}[b^jc^k,a^i]^{\binom{p}{2}}[b^jc^k,a^i,a^i]^{\binom{p}{3}}[b^jc^k,a^i,b^jc^k]^{\frac{p(p-1)(2p-1)}{6}}.
 \end{equation}
 Since $a$ commutes with $[b^j,a]$ and $b$ commutes with $[b,a]$, so we have $[b^jc^k,a^i]=a^{-ijp}[c^k,a^i]$. Thus equation (\ref{HP_1}) reduces to
 \begin{equation}\label{HP_2}
 	(a^ib^jc^k)^p=a^{ip}[c^k,a^i]^{\binom{p}{2}}[c^k,a^i,a^i]^{\binom{p}{3}}[c^k,a^i,b^jc^k]^{\frac{p(p-1)(2p-1)}{6}}.
 \end{equation}
   Now since $b$ and $c$ commutes, so $[c^k,a]=[c,a]^k=b^{-k}$ and as $a^p$ is central element, we must have $a^{-i}b^{-k}a^{i}=b^{-k}a^{ikp}$. By Lemma 3.5 we have $[c^k,a^i]=\prod_{s=0}^{i-1} [c^k,a]^{a^{s}}$, so it follows that $[c^k,a^i]=b^{-ik}a^{kp \binom{i}{2}}$ and therefore equation (\ref{HP_2}) reduces to
   \begin{equation}\label{HP_3}
   	(a^ib^jc^k)^p=a^{ip}[b^{-ik}a^{kp \binom{i}{2}},a^i]^{\binom{p}{3}}[b^{-ik}a^{kp \binom{i}{2}}, b^jc^k]^{\frac{p(p-1)(2p-1)}{6}}.
   \end{equation}   
   Since $a$ commutes with $[b^{-ik},a]$ and $b$ commutes with $[b,a]$, so $[b^{-ik}a^{kp \binom{i}{2}},a^i]=[b^{-ik},a^i]^{a^{kp}\binom{i}{2}}=a^{i^2kp}$. Moreover, $[b^{-ik}a^{kp \binom{i}{2}}, b^jc^k]=[a^{kp \binom{i}{2}}, b^jc^k]=[a^{kp \binom{i}{2}},c^k]=a^{-kp\binom{kp\binom{i}{2}}{2}}$ and thus equation (\ref{HP_3}) reduces to $(a^ib^jc^k)^p=a^{ip}$. So all elements of order $p$ of $P_7$ are in the abelian subgroup $\bigl<a^p,b,c\bigr>$ and therefore no subgroup of $P_7$ is isomorphic to $G_2$, which implies $\text{QIso}(P_7)$ is not a chain for $p\ge 5$. Substituting $p=3$, using GAP \cite{GAP4} we get $\text{QIso}(P_7)$ is a chain. 
   \item Since the nilpotency class of $P_8$ is 3, for $p\ge 5$ from Lemma 3.3 we have that 
 \begin{equation}\label{HP1}
 	(ca^{-1})^p= c^pa^{-p}[a^{-1},c]^{\binom{p}{2}}[a^{-1},c,c]^{\binom{p}{3}}[a^{-1},c,a^{-1}]^{\frac{p(p-1)(2p-1)}{6}}.
 \end{equation}  
 Since $c^p=a^{p}$ and $[a^{-1},c]=ab^{-1}a^{-1}$, equation (\ref{HP1}) reduces to
 \begin{equation}\label{HP2}
 	(ca^{-1})^p=[ab^{-1}a^{-1}, c]^{\binom{p}{3}}[ab^{-1}a^{-1},a^{-1}]^{\frac{p(p-1)(2p-1)}{6}}.
 \end{equation}
 Note that $[ab^{-1}a^{-1},c]=[a,c]^{b^{-1}a^{-1}}[a^{-1},c]=b^{a^{-1}}ab^{-1}a^{-1}=1$. Also $[ab^{-1}a^{-1},a^{-1}]=[b^{-1},a^{-1}]^{a^{-1}}=[b,a]^{(ba)^{-1}a^{-1}}=(a^{-p})^{(ab)^{-1}}$. Substituting these values in equation (\ref{HP2}), we get $(ca^{-1})^p=1$. It now follows that $\bigl<ca^{-1},b,a^p\bigr>\cong G_2$, and therefore $\text{QIso}(P_8)$ is a chain for $p\ge 5$. Substituting $p=3$, using GAP we get $\text{QIso}(P_8)$ is a chain. 
 \item For $p\ge 5$ in $P_9$, we have from Lemma 3.3 that 
 \begin{equation}\label{HP3}
 	(ca^{-\alpha})^p=c^pa^{-\alpha p}[a^{-\alpha},c]^{\binom{p}{2}}[a^{-\alpha},c,c]^{\binom{p}{3}}[a^{-\alpha},c,a^{-\alpha}]^{\frac{p(p-1)(2p-1)}{6}}.
 \end{equation}
 Note that $c^p=a^{\alpha p}$ and $[a^{-\alpha},c]=a^{\alpha}(ab)^{-\alpha}$, Moreover we have
 \begin{equation}\label{HP4}
 	(ab)^{-\alpha}= b^{-\alpha}a^{-\alpha}[a^{-1},b^{-1}]^{\binom{\alpha}{2}}[a^{-1},b^{-1},b^{-1}]^{\binom{\alpha}{3}}[a^{-1},b^{-1},a^{-1}]^{\frac{\alpha(\alpha-1)(2\alpha-1)}{6}}.
 \end{equation}
 Since $[a^{-1},b^{-1}]=[a,b]^{(ab)^{-1}}$ and as $a^p$ is central in $P_9$, equation (\ref{HP4}) reduces to $(ab)^{-\alpha}=b^{-\alpha}a^{-\alpha}(a^p)^{\binom{\alpha}{2}}$ and therefore $[a^{-\alpha},c]=(a^{\alpha}b^{-\alpha}a^{-\alpha})(a^p)^{\binom{\alpha}{2}}$. Substituting the value of $[a^{-\alpha},c]$ in equation (\ref{HP3}) gives 
 \begin{equation}\label{HP5}
 	(ca^{-\alpha})^p= [(a^{\alpha}b^{-\alpha}a^{-\alpha})(a^p)^{\binom{\alpha}{2}},c]^{\binom{p}{3}}[(a^{\alpha}b^{-\alpha}a^{-\alpha})(a^p)^{\binom{\alpha}{2}},a^{-\alpha}]^{\frac{p(p-1)(2p-1)}{6}}.
 \end{equation}
 Note that $[(a^{\alpha}b^{-\alpha}a^{-\alpha})(a^p)^{\binom{\alpha}{2}},a^{-\alpha}]=[b^{-\alpha}, a^{-\alpha}]^{a^{-\alpha}(a^p)^{\binom{\alpha}{2}}}$. Since $a$ commutes with $[b^{-\alpha},a]$ and $b$ commutes with $[b,a]$, we have $[b^{-\alpha}, a^{-\alpha}]^{a^{-\alpha}(a^p)^{\binom{\alpha}{2}}}=([b^{-\alpha},a]^{-\alpha})^{a^{-\alpha}(a^p)^{\binom{\alpha}{2}}}=([b,a]^{\alpha ^2})^{a^{-\alpha}(a^p)^{\binom{\alpha}{2}}}=(a^{-\alpha ^2p})^{a^{-\alpha}(a^p)^{\binom{\alpha}{2}}}$.\par 
 Also note that $[(a^{\alpha}b^{-\alpha}a^{-\alpha})(a^p)^{\binom{\alpha}{2}},c]=([a^{\alpha},c]^{b^{-\alpha}a^{-\alpha}}[a^{-\alpha},c])^{(a^p)^{\binom{\alpha}{2}}}$ and  $[a^{\alpha},c]=a^{-\alpha}(ab)^{\alpha}$ and since Lemma 3.3 implies $(ab)^{\alpha}=a^{\alpha}b^{\alpha}(a^{-p})^{\binom{\alpha}{2}}$, so $[a^\alpha,c ]=b^{\alpha}(a^{-p})^{\binom{\alpha}{2}}$. Therefore, equation (\ref{HP5}) reduces to $(ca^{-\alpha})^p=1$. It now follows that $\bigl<b,ca^{-\alpha}, a^{\alpha p}\bigr>\cong G_2$ and so $\text{QIso}(P_9)$ is a chain for $p\ge 5$. Substituting $p=3$, we get using GAP that $\text{QIso}(P_9)$ is not a chain as no subgroup of $P_9$ is isomorphic to $G_2$.
\item In $P_{10}$, note that $\bigl<a\bigr>\subseteq Z(P_{10})$. So let $g=a^i b^j c^k d^l\in Z(P_{10})$, then $1=[g,b]=[b^jc^kd^l,b]=[d^l,b]=[d,b]^l=a^{-1}$, so $p$ divides $l$. Furthermore, $1=[g,d]=[b^jc^k,d]=[b^j,d]^{c^k}[c^k,d]=(a^j)^{c^k}b^k$, and since $\bigl<a\bigr>\cap \bigl<b\bigr>=1$ so we have that $p$ divides both $j$ and $k$. Therefore $Z(P_{10})=\bigl<a\bigr>$, and $P_{10}$ has a unique normal subgroup of order $p$. Note that $P_{10}/Z(P_{10})\cong G_2 \cong \bigl<a,b,d\bigr>$. The derived subgroup $P_{10}'=\bigl<a,b\bigr>$ is a subgroup of order $p^2$ with $P_{10}/P_{10}' \cong \mathbb{Z}_2 \times \mathbb{Z}_2\cong P_{10}'$. Let $N$ be a normal subgroup of order $p^2$, then by Theorem 1.19 of \cite{isaacs2008finite}, we have $\bigl<a\bigr>\le N$. Since $N/\bigl<a\bigr>$ is normal in $P_{10}/\bigl<a\bigr>$, we have $Z(P_{10}/\bigl<a\bigr>)=\bigl<b \bigl<a\bigr> \bigr>=N/\bigl<a\bigr>$, so $N=\bigl<a,b\bigr>=P_{10}'$. Therefore $P_{10}$ has a unique normal subgroup of order $p^2$ and $\text{QIso}(P_{10})$ is a chain. 
 \end{enumerate} 
 This completes the proof.
	\end{proof}
		Note that GAP computations show that the only non-abelian 2-groups $G$ with order at most $2^4$ for which $\text{QIso}(G)$ is a chain are dihedral groups $D_4$ and $ D_8$, and the semi-dihedral group $SD_{16}=\bigl<a,b \mid a^8=b^2=1, \ b^{-1}ab=b^3\bigr>$. 
\section{Conclusions}
In this paper, we studied the poset $\text{QIso}(G)$ of classes subgroups of a finite group $G$ with isomorphic quotients. In the case of abelian groups, we proved that it is isomorphic to $\text{Iso}(G)^{\text{op}}$. We showed that for dihedral group $\text{QIso}(G)$ is a distributive lattice with the number of elements equal to $\tau (n)+1$. For affine groups, we showed that the problem essentially reduces to cyclic groups. We established that cyclic groups are the only groups with $\text{QIso}(G)\cong \mathcal{L}(G)^{\text{op}}$ and also are the only nilpotent groups with $\text{QIso}(G)\cong \mathcal{N}(G)^{\text{op}}$. Moreover, we established that the only finite $2$-groups all of whose subgroups are q-self dual with $\text{QIso}(G)$ being a chain are either cyclic 2-group, an elementary abelian 2-group or the dihedral group $D_{2^n}$, $n\ge 2$. Lastly, we characterized all $p$-groups $G$ of order at most $p^4$ whose $\text{QIso}(G)$ is a chain.
\section*{Acknowledgment}
The first author acknowledges CSIR (Council of Scientific and Industrial Research) for the financial assistance in the form of junior research fellowship (File No: 09/0414(22038)/2025-EMR-I ).
\bibliographystyle{abbrv}
\bibliography{References}

\begin{thebibliography}{10}

\bibitem{ahmadi20132}
B.~Ahmadi and H.~Doostie.
\newblock On the 2-generator {{\(p\)}}-groups with non-cyclic commutator
  subgroup.
\newblock {\em Azerb. J. Math.}, 4(1):73--78, 2014.

\bibitem{AN201135}
L.~An, J.~Ding, and Q.~Zhang.
\newblock Finite self dual groups.
\newblock {\em J. Algebra}, 341(1):35--44, 2011.

\bibitem{ballal2025posetsclassesautomorphicsubgroups}
S.~Ballal and T.~Halder.
\newblock On {Posets} of {Classes} of {Automorphic} {Subgroups} of {Finite}
  {Groups}.
\newblock \url{https://arxiv.org/abs/2511.00433}, 2025.

\bibitem{Burnside_2012}
W.~Burnside.
\newblock {\em Theory of groups of finite order. 2nd edition.}
\newblock Cambridge University Press., 1911.

\bibitem{easterfield1938extension}
T.~E. Easterfield.
\newblock An extension of a theorem of {Kulakoff}.
\newblock {\em Proc. Camb. Philos. Soc.}, 34:316--320, 1938.

\bibitem{GAP4}
The GAP~Group.
\newblock {\em {GAP -- Groups, Algorithms, and Programming, Version 4.16.1}},
  2026.

\bibitem{isaacs2008finite}
I.~M. Isaacs.
\newblock {\em Finite group theory.}, volume~92 of {\em Grad. Stud. Math.}
\newblock Providence, RI: American Mathematical Society (AMS), 2008.

\bibitem{Janko}
Z.~Janko.
\newblock Finite nonabelian {{\(p\)}}-groups all of whose subgroups are q-self
  dual.
\newblock {\em J. Algebra Appl.}, 13(6):7, 2014.
\newblock Id/No 1450008.

\bibitem{kulakoff}
A.~Kulakoff.
\newblock {\"U}ber die {Anzahl} der eigentlichen {Untergruppen} und der
  {Elemente} von gegebener {Ordnung} in {{\(p\)}}-{Gruppen}.
\newblock {\em Math. Ann.}, 104:778--793, 1931.

\bibitem{Mainardis1998}
M.~Mainardis.
\newblock Finite groups whose poset of conjugacy classes of subgroups is
  isomorphic to the one of an {Abelian} group.
\newblock {\em Boll. Unione Mat. Ital., Sez. B, Artic. Ric. Mat. (8)},
  1(3):691--698, 1998.

\bibitem{Schmidt+1994}
R.~Schmidt.
\newblock {\em Subgroup lattices of groups}, volume~14 of {\em De Gruyter Expo.
  Math.}
\newblock Berlin: Walter de Gruyter, 1994.

\bibitem{shelash}
H.~B. Shelash.
\newblock Counting the number of subgroups and normal subgroups of the group
  {{\(U_{2np}\)}}, {{\(p\)}} is an odd prime.
\newblock {\em Facta Univ., Ser. Math. Inf.}, 34(4):755--760, 2019.

\bibitem{tarnauceanusolitary}
M.~T{\u{a}}rn{\u{a}}uceanu.
\newblock Solitary quotients of finite groups.
\newblock {\em Cent. Eur. J. Math.}, 10(2):740--747, 2012.

\bibitem{tarnauceanu2015posetsclassesisomorphicsubgroups}
M.~T{\u{a}}rn{\u{a}}uceanu.
\newblock The posets of classes of isomorphic subgroups of finite groups.
\newblock {\em Bull. Malays. Math. Sci. Soc. (2)}, 40(1):163--172, 2017.

\bibitem{wilson_1971}
J.~S. Wilson.
\newblock Groups with every proper quotient finite.
\newblock {\em Proc. Camb. Philos. Soc.}, 69:373--391, 1971.

\end{thebibliography}
\end{document}